\documentclass[11pt]{amsart}
\usepackage{graphicx}
\usepackage{color}\usepackage{amsmath,amsthm,amssymb,amsfonts}
\usepackage[english]{babel}
\usepackage{amsmath}
\usepackage{amssymb}
\title{Subespacios densos de $C[0,1]$. Teoremas de Stone-Weierstrass y de Müntz-Szász.}
\author{Mario Pérez Maletzki}
\date{2021}

\newtheorem{teo}{Teorema}[section]
\newtheorem{prop}[teo]{Proposición}

\newtheorem{cor}[teo]{Corolario}

\newtheorem{lema}[teo]{Lema}

\begin{document}
	\maketitle	
	\begin{abstract}
	En este trabajo investigamos algunos resultados acerca de los subespacios densos de $C[0,1]$. Tomando como punto de partida el Teorema de Aproximación de Weierstrass, estudiamos generalizaciones de este en dos direcciones: la primera considerando subespacios dotados con estructura de álgebra, donde el resultado principal será el Teorema de Stone-Weierstrass, y la segunda considerando subespacios generados por conjuntos de monomios cuyos exponentes cumplan ciertas propiedades, donde el resultado central será el Teorema de Müntz-Szász. Finalmente recopilamos los avances más recientes en este tema.
	\end{abstract}
	\begin{abstract}
	In this work, we look into some results about dense subspaces of $C[0,1]$. Being our starting point the Weierstrass' Approximation Theorem, we study generalization of this in two directions: the first one studying subspaces which also have algebra structure, where the main result will be the Stone-Weierstrass theorem, and the second one will be considering subspaces generated by sets of monomials whose exponents satisfy certain properties, where the central result will be the Müntz-Szász theorem. Finally, we gather some of the most recent advances from this topic.

	\end{abstract}
\section{Introducción}

El Teorema de Aproximación de Weierstrass es un resultado de gran importancia en la Teoría de Aproximación. Fue descubierto por Karl Weierstrass en 1885 y afirma que para cada función continua $f:[a,b]\rightarrow\mathbb{C}$ y cada $\varepsilon>0$, existe un polinomio $ p_\varepsilon:\mathbb{R}\rightarrow \mathbb{C}$ tal que $|f(x)-p_\varepsilon(x)|<\varepsilon$ para cada $x\in [a,b]$. 

Este teorema resulta de gran interés tanto práctico como teórico, y no es de extrañar que muchos matemáticos hayan tratado de generalizar este resultado posteriormente. La generalización más conocida es debida al matemático Marshall H. Stone en 1937, quien caracterizó qué subálgebras de funciones continuas son densas en el álgebra de funciones continuas en un espacio compacto arbitrario.

Otra forma de formular el Teorema de Weierstrass cuando $[a,b]=[0,1]$, es que $span\{1,x,x^2,...\}$ es denso en $C([0,1],\mathbb{C})$, donde $span\{1,x,x^2,...\}$ se refiere al espacio vectorial generado por dichos monomios. Con esta formulación del teorema puede resultar natural preguntarse si podemos suprimir algún monomio en $\{1,x,x^2,...\}$ de forma que el espacio vectorial que genere siga siendo denso en $C([0,1],\mathbb{C})$, o de forma más general, preguntarnos qué condiciones tienen que cumplir los exponentes de un conjunto de monomios para que generen un espacio vectorial denso en $C([0,1],\mathbb{C})$.

En 1912 Sergei Bernstein conjeturó que una condición suficiente y necesaria para que una sucesión creciente  $\{\lambda_i\}_{i=0}^\infty$ de números reales  que tiendan a infinito cumpla que $span\{x^{\lambda_i}\mid i=0,1,2...\}$ sea denso en $C([0,1],\mathbb{C})$, es que: $$\sum_{n=0}^\infty  \frac{1}{\lambda_n}=\infty.$$

Dos años más tarde esta conjetura fue probada por el matemático Herman Müntz, y en 1916 Otto Szász dio una demostración simplificada de la de Müntz.

\underline{Notación:} 	Dado un espacio topológico $X$, denotamos por $C(X)$ al espacio de las funciones reales y continuas en $X$, mientras que cuando consideremos las funciones complejas y continuas lo especificaremos usando la notación $C(X,\mathbb{C})$.
Como es habitual denotaremos por $L^2[0,1]$ al espacio de funciones medibles  $f:[0,1]\rightarrow\mathbb{C}$ tales que $$\int_{0}^{1} |f(x)|^2\ dx <\infty.$$

\section{Teorema de Stone-Weierstrass}

En esta sección estdiamos el Teorema de Stone-Weierstrass y como quedan caracterizadas aquellas subálgebras densas en $C[0,1]$, obteniendo como un simple corolario el Teorema de Weierstrass. 

\begin{lema}
	
	Existe una sucesión $(p_n)_{n=0}^\infty$ de polinomios $p_n:\mathbb{R}\rightarrow \mathbb{R}$ cumpliendo $p_n(0)=0$ que converge uniformemente a la función $f(t)=\sqrt t$ en $[0,1]$.
	
\end{lema}

\begin{proof}[Demostración]
	
	Definimos recursivamente los polinomios $p_n$ de la siguiente forma: 	
	\begin{equation}\label{def}	
	p_0(t):=0, \quad p_{n+1}(t):=p_n(t)+ \frac1 2 (t-p^2_n(t)).
	\end{equation}	
	
	Para cada $n\geq0$ es elemental que se cumple la siguiente igualdad: 
	\begin{equation}\label{igualdad}
	\sqrt t -p_{n+1}(t) = (\sqrt t - p_n(t))(1-\frac1 2(\sqrt{t}+p_n(t))).
	\end{equation}
	
	\noindent Probaremos por inducción sobre $n\in \mathbb{N}$ las siguientes afirmaciones: 
	\begin{equation}\label{af1}
	p_n(t)\geq 0, \quad p_n(0)=0, 
	\end{equation}
	
	\begin{equation}\label{af2}
	0\leq \sqrt t - p_n(t) \leq \frac{2\sqrt t}{2+n\sqrt t} \ \ \mbox{para todo } t\in[0,1].
	\end{equation}
	
	Una vez probadas, puesto que $\frac{2\sqrt t}{2+n\sqrt t}\leq \frac{2}{n}$ para cada $t\in [0,1]$,  tomando supremos en $t\in [0,1]$ tendremos que: \begin{equation}\label{af3}
	\sup_{t\in [0,1]} | \sqrt  t - p_n(t) |\leq \frac2 n,
	\end{equation} de donde deduciremos que $p_n(t)$ converge uniformemente a $\sqrt t$ en $[0,1]$.
	\begin{itemize}
		\item El caso $n=0$ está claro, pues $p_0\equiv 0$.
		\item Suponiendo que para $n$ las afirmaciones son ciertas, veamos que también lo son para $n+1$: 
	\end{itemize}

	De (\ref{def}) se deduce trivialmente que $p_{n+1}(0)=0$.
	
	Por hipótesis de inducción tenemos que $p_n(t)\geq 0$ y $0\leq \sqrt t - p_n(t)$ para cada $t\in [0,1]$, por lo que: $$p_{n+1}(t)=p_n(t)+ \frac1 2 (t-p^2_n(t))=p_n(t)+ \frac1 2 (\sqrt{t}-p_n(t))(\sqrt{t}+p_n(t))\geq 0.$$

	Por (\ref{igualdad}) y por la hipótesis de inducción tenemos que para $0\leq t \leq1$: $$ \sqrt t -p_{n+1}(t)= (\sqrt t - p_n(t))(1-\frac1 2(\sqrt{t}+p_n(t)))\geq (\sqrt t - p_n(t))(1-\sqrt{t}) \geq 0,$$

	$$\sqrt t -p_{n+1}(t)= (\sqrt t - p_n(t))(1-\frac1 2(\sqrt{t}+p_n(t)))\leq \frac{2\sqrt t}{2+n\sqrt t}\cdot \frac1 2 (2-\sqrt t -p_n(t))$$ $$\leq \frac{2\sqrt t}{2+n\sqrt t}\cdot \frac1 2 (2-\sqrt t)\leq \frac{2\sqrt t}{2+(n+1)\sqrt t}.$$

	Por tanto, queda probado el lema.

\end{proof}
\begin{lema}\label{vabs}
	
	Dado $a>0$, existe una sucesión $(q_n)_{n=0}^\infty$ de polinomios $q_n:\mathbb{R}\rightarrow \mathbb{R}$ con $q_n(0)=0$ que converge uniformemente a $f(t)= |t|$ en $[-a,a]$.
	
\end{lema}

\begin{proof}[Demostración]
	
	Para cada $n\geq0$ consideramos el polinomio $p_n$ definido en (\ref{def}) y definimos: $$q_n(t):=a\cdot p_n\left(\frac{t^2}{a^2}\right).$$ 
	
	Se tiene que para cada $t\in [-a,a]$: $$| |t|-q_n(t)|=\left|a\sqrt{\frac{t^2}{a^2}}- q_n(t)\right| =a\left| \sqrt{\frac{t^2}{a^2}}- p_n\left(\frac{t^2}{a^2}\right)\right|\leq \frac{2a}{n}$$ siendo la última desigualdad consecuencia de (\ref{af3}).
	
	\noindent De aquí, queda probado que $q_n(t)$ converge uniformemente a $|t|$ en $[-a,a]$.
	
\end{proof}

	Con las operaciones $f+\lambda g:= t\mapsto f(t)+\lambda g(t)$ y $f\cdot g:= t\mapsto f(t)g(t)$ tenemos que $C(X)$ es un álgebra.

	Si $D\subseteq C(X)$, denotaremos por $A(D)$ al álgebra generada por $D$, es decir, la menor subálgebra de $C(X)$ que contiene a $D$.
	
	Es bien sabido que $A(D)$ coincide con el siguiente conjunto: $$ \left\{\sum_{0\leq v_1,\ldots,v_r\leq n} a_{v_1\ldots v_r}d_1^{v_1}\ldots d_r^{v_r} \mid r,n\in\mathbb{N}, \ d_i\in D \ (1\leq i\leq r),\ a_{0\ldots 0}=0\right\},$$
	donde conforme a nuestra notación $ \ a_{v_1\ldots v_r}\in\mathbb{R}$.

\begin{lema} \label{lemautil}Sea $X$ un espacio compacto y $A$ una subálgebra cerrada de $C(X)$. Si $f,g\in A$ entonces $|f|, \ \max\{f,g\}, \ \min\{f,g\} \in A$.
	
\end{lema}

\begin{proof}[Demostración]
	
	Debido a que: $$\min\{f,g\}=\frac{f+g-|f-g|}{2}\ \ \mbox{ y }\ \max\{f,g\}=\frac{f+g+|f-g|}{2},$$ bastará con demostrar que si $f\in A$ entonces $|f|\in A$.

	Sea $a:=\sup\{|f(x)| \mid x\in X\}$. Por ser $f$ continua y $X$ compacto se tiene que $a<\infty$, y por el Lema \ref{vabs} tenemos que para cada $\varepsilon>0$ existe un polinomio $p_\varepsilon$ tal que $p_\varepsilon(0)=0$ y $\left| |f(x)| - p_\varepsilon(f(x))\right|<\varepsilon$ para cada $ x\in X$. 
	
	Tenemos que $ p_\varepsilon\circ f\in A$, pues al ser $p_\varepsilon(0)=0$ el polinomio $p_\varepsilon$ carece de término independiente, y por tanto $ p_\varepsilon\circ f$ puede escribirse como un elemento de $A(D)$. Además $ p_\varepsilon\circ f$ converge uniformemente a $|f|$ en $X$, luego por ser $A$ cerrado se deduce que $|f|\in A$ y así, queda demostrado el resultado.
	
\end{proof}

\begin{lema}
	
	Sea $X$ compacto y $A$ una subálgebra de $C(X)$. Si $f,g\in \bar{A}$ y $c\in\mathbb{R}$ entonces $f+g, \ f\cdot g$ y $c\cdot f$  pertenecen a $\bar{A}$. Por tanto, se tiene que $\bar{A}$ es también una subálgebra de $C(X)$.
	
\end{lema}

\begin{proof}[Demostración]
	
	Como $f,g\in \bar{A}$, existen sucesiones $(f_n),(g_n)\subseteq A$ tales que $f_n\rightarrow f$  y  $g_n\rightarrow g$ uniformemente en $X$. Por tanto $(f_n + g_n)\rightarrow f + g$, $(f_n\cdot g_n)\rightarrow f\cdot g$ y $(c\cdot f_n)\rightarrow c\cdot f$ uniformemente en $X$. Como además $A$ es un álgebra tenemos que $f_n\cdot g_n ,\  f_n + g_n$ y $c\cdot f_n\in A$ para cada $n\in\mathbb{N}$ y se sigue el resultado.
	
\end{proof}

\begin{teo}[Stone-Weierstrass]
	
	Sea $X$ un espacio compacto y $D\subseteq C(X)$ tal que:
	
	\begin{enumerate}
		
		\item Para cada $x\in X$ existe $f_x\in D$ tal que $f_x(x)\neq 0$.\label{prim}
		
		\item Para cada par de elementos distintos $x,y\in X$, existe $f\in D$ tal que $f(x)\neq f(y)$.\label{seg}
		
	\end{enumerate}

	Entonces $\overline{A(D)}=C(X)$.
	
\end{teo}

\begin{proof}[Demostración]
	
	Sean $f\in C(X)$ y $\varepsilon>0$. Vamos a demostrar que existe una función $g_\varepsilon \in \overline{A(D)}$ tal que $| f(x)- g_\varepsilon(x)|< \varepsilon$ para todo $x\in X$, con lo que tendremos que: $$\overline{A(D)} =\overline{\ \overline{A(D)}\ } =C(X).$$

	Primero demostraremos que para cada par de elementos $ y,z\in X$, existe una función $h\in A(D)$ tal que $h(y)=f(y)$ y $h(z)=f(z)$. En efecto, por (\ref{prim})  existen funciones $ f_y ,f_z \in D$ tales que $f_y(y)\neq 0$ y $f_z(z)\neq 0$. Definimos las funciones $f_1,f_2:X\rightarrow\mathbb{R}$ como: $$f_1:= \frac{1}{f_y(y)}\cdot f_y , \ \ f_2:= \frac{1}{f_z(z)}\cdot f_z.$$ 
	
	Entonces  $f_1, f_2\in A(D)$, $f_1(y)=1$ y $f_2(z)=1$.

	Sea $h_1:=f_1 + f_2 - f_1\cdot f_2$; es claro que $h_1\in A(D)$ y que $h_1(y)=h_1(z)=1$.

	Si $y=z$, definiendo $h:=f(y)\cdot h_1$ habremos acabado. 
	
	Si $y\neq z$, por la condición (\ref{seg}) existe una función $h_2 \in D$ tal que $h_2(y)\neq h_2(z)$ y entonces definimos: $$h:=\frac{f(y)-f(z)}{h_2(y)-h_2(z)}\cdot h_2 - \frac{f(y)h_2(z)- f(z)h_2(y)}{h_2(y)-h_2(z)}\cdot h_1.$$ Así, $h\in A(D)$ y además cumple que $h(y)=f(y)$ y $h(z)=f(z)$.

	A continuación demostramos que dado $z\in X$, existe $ h_z \in \overline{A(D)}$ tal que $h_z(z)=f(z)$ y $h_z(x)<f(x) +\varepsilon$ para cada $x\in X$. 
	Por lo anterior tenemos que para todo $y \in X$ existe una función $h_y\in A(D)$ tal que $h_y(y)=f(y)$ y $h_y(z)=f(z)$. Como $f-g_y$ es continua y se anula en $y$, existe un entorno $U_y$ de $y$ tal que $g_y(x)<f(x)+ \varepsilon$ para cada $x\in U_y$. La compacidad de $X$ implica que existe un conjunto finito de puntos $L\subseteq X$ tal que $X=\bigcup _{y\in L} U_y$. Tomando $h_z:= \min\{g_y \mid y\in L\}$ por el Lema \ref{lemautil} se tiene que $h_z\in \overline{A(D)}$. Además $h_z(z)=\min \{g_y(z) \mid y\in L\}=z$ y si $x\in U_y$ entonces $h_z(x)<g_y(x)<f(x)+ \varepsilon$, por lo que $h_z(x)<f(x)+\varepsilon$ para todo $x\in X$.

	Finalmente elijamos para cada $z\in X$ una función $h_z$ con estas propiedades. Como $f-h_z$ es continua y se anula en $z$ tenemos que existe un entorno $ W_z$ de $z$ tal que $h_z(x)>f(x)- \varepsilon$ para cada $x\in W_z$. Como $X$ es compacto existirá un conjunto finito de puntos $ K\subseteq X$ tal que $X =\bigcup _{z\in K} W_z$. 
	
	Definimos $g_\varepsilon:= \max\{h_z \mid z\in K\}$ y tendremos por el Lema \ref{lemautil} que $g_\varepsilon\in \overline{A(D)}$. Además si $x\in W_z$ entonces $f(x)- \varepsilon < h_z(x)<g_\varepsilon(x)<f(x)+ \varepsilon$ y se sigue el resultado.

\end{proof}

\begin{cor}\label{weierstrass-real}
	
	Sea $f:[a,b]\rightarrow \mathbb{R}$ una función continua. Entonces dado $\varepsilon>0$ existe un polinomio $ p_\varepsilon:\mathbb{R}\rightarrow \mathbb{R}$ tal que $|f(x)-p_\varepsilon(x)|<\varepsilon$ para cada $x\in [a,b]$.
	
\end{cor}

\begin{proof}[Demostración]
	
	Como $[a,b]$ es compacto y el álgebra de los polinomios coincide con el álgebra generada por $D=\{1,x\}$ (conjunto que cumple las hipótesis del Teorema de Stone-Weierstrass trivialmente) el resultado se sigue de forma inmediata.
	
\end{proof}

\begin{teo}[Weierstrass]\label{weierstrass}
	Dada una función continua $f:[a,b]\rightarrow \mathbb{C}$ y un $\varepsilon>0$, existe un polinomio $P:\mathbb{R}\rightarrow \mathbb{C}$ tal que $|f(x)-P(x)|<\varepsilon$ para cada $x\in [a,b]$.
\end{teo}
\begin{proof}[Demostración]
	Aplicando el Corolario \ref{weierstrass-real} a la parte real e imaginaria de $f$ tendremos que existen polinomios $p_\varepsilon,q_\varepsilon: \mathbb{R}\rightarrow \mathbb{R}$ tales que: $$|\Re(f)(x)-p_\varepsilon(x)|<\varepsilon/2\ \ \  \mbox{ y }\ \ |\Im(f)(x)-q_\varepsilon(x)|<\varepsilon/2$$ para todo $x\in[a,b]$. El polinomio $P:\mathbb{R}\rightarrow \mathbb{C}$ definido como $P(x):=p_\varepsilon(x)+iq_\varepsilon(x)$,
	cumple que: $$|f(x)-P(x)|<\varepsilon \quad \mbox{ para cada }x\in[a,b].$$ 
\end{proof}
Una consecuencia importante de este resultado es que los polinomios con coeficientes cuyas partes reales e imaginarias sean racionales también serán densos en $C([a,b],\mathbb{C})$, y que por tanto este espacio es separable. Observar que si sustituimos en el conjunto $D$ en la demostración del Corolario \ref{weierstrass-real} el monomio $x$ por alguna potencia $x^n$, las hipótesis del Teorema de Stone-Weierstrass siguen siendo ciertas.

\begin{cor}\label{potencias de n}
	Dada una función continua $f:[a,b]\rightarrow\mathbb{C}$ y un entero positivo $n$, existe una sucesión de polinomios $(p_n)_{n=1}^\infty$ cuyos monomios cumples que sus exponentes son múltiplos de $n$ y que cumple que $(p_n)_{n=1}^\infty$ converge uniformemente a $f$ en $[a,b]$.
\end{cor}

\noindent Recordamos que en el caso real llamamos polinomio trigonométrico a cualquier combinación lineal de las funciones $\sin(nx)$ y $\cos(mx)$ siendo $n$ y $ m$ números naturales. 

\noindent El siguiente resultado es de vital importancia en Análisis Armónico. 
\begin{cor}
	
	Toda función real continua en $[-\pi,\pi]$ puede ser aproximada uniformemente por polinomios trigonométricos.
	
\end{cor}

\begin{proof}[Demostración]
	
	Como $D=\{\sin(x),\cos(x)\}$ cumple las hipótesis del Teorema de Stone-Weierstrass y $A(D)$ contiene todos los polinomios trigonométricos se sigue el resultado.
	
\end{proof}

\begin{cor}
	
	Sea $X\subseteq \mathbb{R}^n$ un conjunto compacto. Entonces toda función continua $f:X\rightarrow\mathbb{C}$ puede ser aproximada uniformemente por polinomios.
	
\end{cor}

\begin{proof}[Demostración]
	Basta con probarlo para $f:X\rightarrow\mathbb{R}$ y luego aplicar el resultado a las partes real e imaginaria. 
	
	Supongamos pues que $f:X\rightarrow\mathbb{R}$. 
	Si $f_0\equiv 1$ y definimos para $0\leq i\leq n$ las proyecciones $f_i:\mathbb{R}^n\rightarrow\mathbb{R}$ como $f_i(x_1,x_2,\ldots,x_n):=x_i$ tenemos que el álgebra generada por $\smash{D=\{f_0,f_1,\ldots,f_n\}}$ coincide con el álgebra de los polinomios en $\mathbb{R}^n$, y como $D$ cumple las hipótesis del Teorema de Stone-Weierstrass se sigue el resultado.
	
\end{proof}

El Teorema de Weierstrass \ref{weierstrass} es falso si sustituimos la función $f$ por una función continua cuyo dominio es un conjunto compacto del plano complejo, pues es bien sabido que el límite uniforme de funciones holomorfas es una función holomorfa y por tanto si consideramos una función continua pero no holomorfa, como por ejemplo la función conjugación, esta no puede ser aproximada por polinomios complejos. Es necesario pues incluir hipótesis adicionales para el caso complejo.

\begin{teo}[Stone-Weierstrass, versión compleja]
		
	Sea $X$ un espacio compacto y $D\subseteq C(X,\mathbb{C})$ tal que:
	
	\begin{enumerate}
		
		\item Para cada $x\in X$ existe $f_x\in D$ tal que $f_x(x)\neq 0$.
		
		\item Para cada par de puntos distintos $x,y\in X$, existe $f\in D$ tal que $f(x)\neq f(y)$.
		
		\item Si $f\in D$, entonces la función conjugada de $f$ también pertenece a $D$, i.e. $\bar{f}\in D$.
	\end{enumerate}

	Entonces $\overline{A(D)}=C(X,\mathbb{C})$.
	
\end{teo}
\begin{proof}[Demostración]
	Consideremos el conjunto $D_{\mathbb{R}}:=\{f\in D\mid f(X)\subset\mathbb{R} \}$.
	
	Si $f\in D$ entonces $f=\Re(f)+i\Im(f)$, donde está claro que $\smash{\Re(f)(X)\subset\mathbb{R}}$ y $\Im(f)(X) \subset\mathbb{R}$.
	
	Por hipótesis también $\bar{f}=\Re(f)-i\Im(f)\in D$, luego $\Re(f), \Im(f)\in D$ y por definición de $D_{\mathbb{R}}$ tendremos que $\Re(f), \Im(f)\in D_{\mathbb{R}}$.
	
	Bastará con probar que el conjunto $D_{\mathbb{R}}$ cumple las condiciones del teorema de Stone-Weierstrass, pues entonces si $f\in C(X,\mathbb{C})$ y $\varepsilon>0$, podremos encontrar elementos $p,q\in A(D_{\mathbb{R}})$ tales que para cada $x\in X$: $$|\Re(f)(x)-p(x)|<\frac{\varepsilon}{2}\quad \mbox{ y }\quad |\Im(f)(x)-q(x)|<\frac{\varepsilon}{2}$$
	y por tanto $g:=p+iq\in A(D)$ cumplirá que $|f(x)-g(x)|<\varepsilon$ para cada $x\in X$, de donde se seguirá el resultado.
	
	Pues bien, dado $x\in X$ existe por hipótesis una función $h\in D$ tal que $h(x)\neq 0$, pero esto es cierto si y sólo si no se anulan simultáneamente en $x$ las funciones $\Re(h)$ y $\Im(h)$, las cuales hemos probado que pertenecen a $D_{\mathbb{R}}$. Por tanto $D_{\mathbb{R}}$ cumple la primera hipótesis del teorema de Stone-Weierstrass.
	
	Finalmente, si $x\neq y$ son elementos de $X$, existe una función $h\in D$ tal que $h(x)\neq h(y)$, lo cual nuevamente solo se puede cumplir si no se dan simultáneamente las condiciones $$\Re(h)(x)=\Re(h)(y)\quad \mbox{y}\quad \Im(h)(x)=\Im(h)(y)$$ por lo que $D_{\mathbb{R}}$ también cumple la segunda hipótesis del teorema de Stone-Weierstrass.
	
	\end{proof}
\section{Teorema de Müntz-Szász}

En esta sección estudiamos un bello resultado en Teoría de la Aproximación: el Teorema de Müntz-Szász. Además, investigamos algunas consecuencias de este resultado.

Comenzamos con unos resultados preparatorios, necesarios para la demostración del Teorema Principal.

\begin{prop}\label{prop-prod}
	
	Sea $\{a_n\}_{n=1}^\infty$ una sucesión de números reales $a_n>1$ para cada $n\in\mathbb{N}$ tal que $\smash{a_n \rightarrow \infty}$. Entonces: $$\prod\limits_{i=1}^n \left(1-\frac{1}{a_i}\right)\stackrel{n\rightarrow\infty}\longrightarrow 0\  \mbox{ si y sólo si }\  \sum\limits_{i=1}^\infty \frac{1}{a_i}=\infty .$$
	
\end{prop}

\begin{proof}[Demostración]
	
	Tenemos que: 
	
	\begin{equation}\label{div1}	
	\prod_{i=1}^n\left(1-\frac{1}{a_i}\right)\rightarrow 0\ \mbox{ si y sólo si }\ \sum_{i=1}^n\log\left(1+\frac{1}{a_i - 1}\right) \rightarrow \infty.
	\end{equation}

	Puesto que $$\lim_{x\rightarrow 0} \frac{\log(1+x)}{x} = 1,$$

	dado $\varepsilon := \frac{1}{2}$, existirá un $\delta >0$ tal que si $0<|x|<\delta$ entonces $$\left|\frac{\log(1+x)}{x} - 1\right|<\frac{1}{2}.$$

	Como $\smash{a_n \rightarrow \infty}$, existe $ n_0 \in \mathbb{N}$ tal que para cada $ n>n_0$ sea  $\delta>\frac1{a_n}>0$.

	Por tanto, para cada $n>n_0$: $$\frac{1}{2}\cdot \frac1{a_n}<\log\left(1+\frac1{a_n}\right)<\frac{3}{2}\cdot \frac1{a_n},$$
	
	de lo que se deduce que: 
	
	\begin{equation} \label{div2}	
	\sum\limits_{i=1}^\infty \frac{1}{a_i}=\infty\ \mbox{ si y sólo si }\ \sum_{i=1}^\infty \log\left(1+\frac1{a_i}\right)=\infty	
	\end{equation}

	Juntando (\ref{div1}) y (\ref{div2}) se sigue el resultado.
	
\end{proof}

\begin{lema} [Gram]\label{Gram}
	
	Sea $\mathcal{H}$ un espacio de Hilbert (complejo), $g\in \mathcal{H}$ y $V:=span\{f_1,\ldots , f_n\}$ un subespacio de $\mathcal{H}$ de dimensión $n$. Entonces, $$d(g,V)^2= \frac{G(f_1,\ldots, f_n, g)}{G(f_1,\ldots ,f_n)}$$

	\noindent siendo $G(h_1,\ldots , h_r)$ el determinante de la matriz de Gram asociada a los elementos $h_1, \ldots ,h_r \in \mathcal{H}$; es decir $$ G(h_1,\ldots , h_r):=\det\left ( 	
	\begin{array}{ccc}
		
		\langle h_1,h_1\rangle & \ldots & \langle h_1,h_r \rangle \\
		
		\vdots & & \vdots \\
		
		\langle h_r,h_1 \rangle & \ldots & \langle h_r,h_r \rangle
		\end{array}
		\right ).$$	
	Notar que $h_1,\ldots,h_r$ son linealmente independientes si y sólo si $\smash{G(h_1,\ldots,h_r)\neq0}$.
	
\end{lema}

\begin{proof}[Demostración]
	
	Por el Teorema de proyección de Riesz, tenemos que existe un único elemento $f\in V$ tal que $\smash{d(g,V)^2=\langle g-f,g-f\rangle}$ y este elemento además cumple que $g-f\in V^\perp$. Por definición de $V$ sabemos que existen unos únicos $c_1,\ldots,c_n\in\mathbb{C}$ tal que $f=\sum\limits_{i=1}^n c_i f_i$, luego se cumplen las siguientes ecuaciones $$ \langle f_i,g-f\rangle=0 \quad 0\leq i\leq n$$
	
	Usando el hecho de que $d(g,V)^2=\langle g,g\rangle - \langle g,f\rangle$ (por ser $g-f\in V^\perp$) y juntando todas las ecuaciones, si las desarrollamos usando las propiedades del producto escalar nos queda el siguiente sistema:
	
	$$\left\{ \begin{array}{cclrl}
	
	\bar{c_1}\langle f_1,f_1\rangle + &\ldots &+ \bar{c_n}\langle f_1,f_n\rangle &- \langle f_1,g\rangle & =0 \\
	
	\vdots & & \vdots & &\\
	
	\bar{c_1}\langle f_n,f_1\rangle +& \ldots &+ \bar{c_n}\langle f_n,f_n\rangle &-\langle f_n,g\rangle & =0\\
	\bar{c_1}\langle g, f_1\rangle +& \ldots &+ \bar{c_n}\langle g,f_n\rangle &+d(g,V)^2 -\langle g,g\rangle & =0
	
\end{array}\right.$$

Entonces, como el vector no nulo $(\bar{c_1},\ldots,\bar{c_n},-1)$ es también solución del sistema homogéneo anterior, será $$\det \left( 
\begin{array}{cccc}
	
	\langle f_1,f_1\rangle & \ldots & \langle f_1,f_n\rangle & -\langle f_1,g\rangle\\
	
	\vdots & & \vdots & \vdots\\
	
	\langle f_n,f_1 \rangle & \ldots & \langle f_n,f_n \rangle & -\langle f_n,g\rangle \\
	
	\langle g,f_1 \rangle & \ldots & \langle g,f_n \rangle & d(g,V)^2 - \langle g,g \rangle
	
\end{array}
\right )=0$$

Escribiendo la última columna como $$(0,\ldots,0,d(g,V)^2)-(\langle f_1,g\rangle,\ldots,\langle f_n,g\rangle,\langle g,g\rangle)$$ deducimos que: $$d(g,V)^2\cdot G(f_1,\ldots ,f_n)= G(f_1,\ldots ,f_n,g)$$ y de aquí se sigue el resultado.

\end{proof}

El siguiente es un resultado técnico cuya prueba sólo utiliza operaciones elementales de los determinantes para llegar a la igualdad deseada.

\begin{lema}[Cauchy]\label{Cauchy}

Si $n\in\mathbb{N}$ y  $x_1,\ldots ,x_n,y_1,\ldots ,y_n \in \mathbb{C}$ cumplen que $x_i+y_j\neq 0$ para cada $ 0\leq i,j\leq n$, entonces: $$D_n:=\det\left (
\begin{array}{ccc}
	\frac1{x_1+y_1} & \ldots & \frac1{x_1+y_n} \\
	
	\vdots & & \vdots \\
	
	\frac1{x_n+y_1} & \ldots & \frac1{x_n+y_n}
	\end{array}
	\right )=\frac{\prod\limits_{1\leq i<j\leq n}(x_j-x_i)(y_j-y_i)}{\prod\limits_{1\leq i,j\leq n}(x_i+y_j)}.$$

\end{lema}

\begin{proof}[Demostración]

Primero sacamos común denominador en cada columna y extraemos dicho factor de cada columna fuera del determinante obteniendo: $$D_n=\frac{1}{\prod\limits_{1\leq i,j\leq n} (x_i+y_j)}\det\left(
\begin{array}{ccc}
	
	\prod_{i\neq 1}(x_i+y_1) & \ldots & \prod_{i\neq 1}(x_i+y_n)\\
	
	\vdots & & \vdots\\
	
	\prod_{i\neq n}(x_i+y_1) & \ldots & \prod_{i\neq n}(x_i+y_n)
	
\end{array}
\right).$$

A continuación, restamos en la primera fila de la anterior matriz la segunda fila, a la segunda fila la tercera, y así hasta restar a la ($n-1$)-ésima fila la última, dejando esta igual. Sacando los factores comunes obtenemos: $$D_n=\frac{\prod\limits_{i=2}^n (x_i-x_{i-1})}{\prod\limits_{1\leq i,j\leq n} (x_i+y_j)}\det\left(
\begin{array}{ccc}
	
	\prod_{i\neq 1,2}(x_i+y_1) & \ldots & \prod_{i\neq 1,2}(x_i+y_n)\\
	
	\vdots & & \vdots\\
	
	\prod_{i\neq n-1,n}(x_i+y_1) & \ldots & \prod_{i\neq n-1,n}(x_i+y_n)\\
	
	\prod_{i\neq n}(x_i+y_1) & \ldots & \prod_{i\neq n}(x_i+y_n)	
\end{array}
\right).$$

Ahora repetimos el mismo proceso pero dejando intactas las dos últimas filas, luego dejando intactas las 3 últimas y así sucesivamente hasta que llegamos a la siguiente expresión: $$ D_n=\frac{\prod\limits_{1\leq i<j\leq n}(x_j-x_i)}{\prod\limits_{1\leq i,j\leq n}(x_i+y_j)} \left(
\begin{array}{ccc}
	
	1 & \ldots & 1\\
	
	(x_1+y_1) & \ldots & (x_1+y_n)\\
	
	\vdots & & \vdots\\
	
	\prod_{i\neq n}(x_i+y_1) & \ldots & \prod_{i\neq n}(x_i+y_n)	
\end{array}
\right). $$

Denotando por $R_i$ la fila $i$-ésima, en el siguiente paso realizamos las siguientes operaciones (en el mismo orden): $R_n-(x_{n-1}+y_1)R_{n-1}, R_{n-1}-(x_{n-2}+y_1)R_{n-2}, \ldots, R_2-(x_1+y_1)$, y así llegamos a que, definiendo la matriz $$A_n:= \left(
\begin{array}{cccc}

1 & 1 & \ldots & 1\\

0 & (y_2-y_1) & \ldots & (y_n-y_1)\\

\vdots & & & \vdots\\

0 & (y_2-y_1)\prod_{i=1}^{n-2}(x_i+y_1) & \ldots & (y_n-y_1)\prod_{i=1}^{n-2}(x_i+y_n)

\end{array}
\right),$$ se cumple que 
$$ D_n=\frac{\prod\limits_{1\leq i<j\leq n}(x_j-x_i)}{\prod\limits_{1\leq i,j\leq n}(x_i+y_j)}\cdot  A_n.$$

Si ahora desarrollamos el determinante por la primera columna y sacamos los factores comunes de cada columna, observamos que repitiendo el proceso anterior $n-2$ veces más llegamos a la igualdad deseada.

\end{proof}

\begin{lema}\label{lema-distancia}

Sean $q,\lambda_0,\ldots,\lambda_n$ números reales mayores estrictamente que $-\frac{1}{2}$. Entonces si $d$ es la distancia inducida por la norma del espacio de Hilbert $L^2[0,1]$, se tiene que: $$\delta:=d(x^q,span_{L^2[0,1]}\{x^{\lambda_0},\ldots,x^{\lambda_n}\})=\frac1{\sqrt{2q+1}}\prod\limits_{i=0}^n \left|\frac{q-\lambda_i}{q+\lambda_i+1}\right| .$$

\end{lema}

\begin{proof}[Demostración]

Por el Lema de Gram (\ref{Gram}), sabemos que $$\delta=\sqrt{\frac{G(x^{\lambda_0},\ldots,x^{\lambda_n},x^q)} {G(x^{\lambda_0},\ldots,x^{\lambda_n})}}.$$

Además, en $L^2[0,1]$ se tiene que si $a,b>-\frac{1}{2}$, $$\langle x^a,x^b\rangle=\int^1_0 \!\! x^a x^b\, dx = \frac1{a+b+1}$$

Así que los determinantes de las matrices de Gram serán de la siguiente forma: $$G(x^{\lambda_0},\ldots,x^{\lambda_n})=\det \left(
\begin{array}{ccc}
	
	\frac1{\lambda_0+\lambda_0+1} & \ldots & \frac1{\lambda_0+\lambda_n+1}\\
	
	\vdots & & \vdots\\
	
	\frac1{\lambda_n+\lambda_0+1} & \ldots & \frac1{\lambda_n+\lambda_n+1}	
\end{array}
\right) $$

$$G(x^{\lambda_0},\ldots,x^{\lambda_n},x^q)=\det \left(
\begin{array}{cccc}
	
	\frac1{\lambda_0+\lambda_0+1} & \ldots & \frac1{\lambda_0+\lambda_n+1} & \frac1{\lambda_0+q+1}\\
	
	\vdots & & \vdots\\
	
	\frac1{\lambda_n+\lambda_0+1} & \ldots & \frac1{\lambda_n+\lambda_n+1}& \frac1{\lambda_n+q+1}\\
	
	\frac1{q+\lambda_0 +1} & \ldots & \frac1{q+\lambda_n+1} & \frac1{q+q+1}
	
\end{array}
\right) $$

Por tanto, usando el Lema de Cauchy \ref{Cauchy} con $x_i=\lambda_i, \ y_i=\lambda_i+1$ si $0\leq i\leq n$ y $x_{n+1}=q, \ y_{n+1}=q+1$ y simplificando llegamos al resulto deseado.

\end{proof}

\underline{Observación:} Naturalmente, podríamos obtener este resultado ortonormalizando el conjunto $\{x^{\lambda_0},\ldots,x^{\lambda_n}\}$ en $L^2[0,1]$ vía el método de Gram-Schmidt, calculando la proyección de $x^q$ en el conjunto resultante y hallar la distancia entre $x^q$ y su proyección. Sin embargo este método habría sido poco práctico, y por ello hemos probado el resultado de esta forma. 

\begin{teo}[Müntz-Szász]

Sea $\{\lambda_0\!=0\!<\lambda_1<\lambda_2<\ldots\}$ una sucesión de números reales tales que $lim_{j\rightarrow\infty}\lambda_j=\infty$. Entonces, $\smash{span_{C([0,1],\mathbb{C})}\{x^{\lambda_i}\mid i\geq0 \}}$ es denso en $C([0,1],\mathbb{C})$ si y sólo si: $$\sum_{i=0}^\infty \frac1{\lambda_i}=\infty.$$

\end{teo}

\begin{proof}[Demostración]

	Comenzamos probando la siguiente afirmación:
	
	\begin{itemize}
		\item Dado $\smash{q\in \mathbb{N}\setminus \{\lambda_i \mid i\in\mathbb{N}\},}$ se tiene que $$x^q\in \overline{span_{L^2[0,1]}\{x^{\lambda_i}\mid i\in \mathbb{N}\cup \{0\} \}}\ \mbox{ si y sólo si }\ \sum_{i=0}^\infty \frac1{\lambda_i}=\infty.$$
	\end{itemize}
	
	En efecto, por un lado $$x^q\in \overline{span_{L^2[0,1]}\{x^{\lambda_i}\mid i\in \mathbb{N}\cup \{0\} \}}\ \mbox{ si y sólo si }\  d(x^q,span_{L^2[0,1]}\{x^{\lambda_0},\ldots,x^{\lambda_n}\})\rightarrow0$$ cuando $n\rightarrow\infty$. 
	
	Por el Lema \ref{lema-distancia} tenemos que $$d(x^q,span_{L^2[0,1]}\{x^{\lambda_0},\ldots,x^{\lambda_n}\})=\frac1{\sqrt{2q+1}}\prod\limits_{i=0}^n \left|\frac{q-\lambda_i}{q+\lambda_i+1}\right| $$
	y puesto que $\left|\frac{q-\lambda_i}{q+\lambda_i+1}\right| =\left|1-\frac{2q+1}{q+\lambda_i+1}\right| $, afirmamos por la Proposición \ref{prop-prod} que $$d(x^q,span_{L^2[0,1]}\{x^{\lambda_0},\ldots,x^{\lambda_n}\})\rightarrow0\ \mbox{ si y sólo si }\ \sum_{i=0}^\infty \frac{2q+1}{q+\lambda_i+1}=\infty .$$
	Es trivial que $\sum_{i=0}^\infty \frac{2q+1}{q+\lambda_i+1}=\infty \ \mbox{ si y sólo si }\  \sum_{i=0}^\infty \frac1{\lambda_i}=\infty,$ luego queda probado que $$x^q\in \overline{span_{L^2[0,1]}\{x^{\lambda_i}\mid i\in \mathbb{N}\cup \{0\} \}} \ \mbox{ si y sólo si }\  \sum_{i=0}^\infty \frac1{\lambda_i}=\infty.$$

	Veamos ahora que si $\smash{span_{C([0,1],\mathbb{C})}\{x^{\lambda_i}\mid i\in \mathbb{N}\cup \{0\} \}}$ es denso en $C([0,1],\mathbb{C})$, entonces $\sum_{i=0}^\infty \frac1{\lambda_i}=\infty.$
	
	Por un lado, para cada función $f\in C([0,1],\mathbb{C})$ se tiene la desigualdad $$||f||_{L^2[0,1]}=\sqrt{\int_{0}^{1}(f(x))^2\ dx} \leq ||f||_{\infty},$$ luego si $\smash{span_{C([0,1],\mathbb{C})}\{x^{\lambda_i}\mid i\in \mathbb{N}\cup \{0\} \}}$ es denso en $C[0,1]$ con la norma del supremo, puesto que $C([0,1],\mathbb{C})$ es denso en $L^2[0,1]$ con la norma $||\cdot||_{L^2[0,1]}$, tendremos que $\smash{span_{L^2[0,1]}\{x^{\lambda_i}\mid i\in \mathbb{N}\cup \{0\} \}}$ es denso en $L^2[0,1]$ con su correspondiente norma. 
	
	Como consecuencia, tendremos que para cada $\smash{q\in \mathbb{N}\setminus \{\lambda_i \mid i\in\mathbb{N}\},}$ se tiene que $$x^q\in \overline{span_{L^2[0,1]}\{x^{\lambda_i}\mid i\in \mathbb{N}\cup \{0\} \}}$$ lo cual, según hemos visto al principio de la demostración, implica que $\sum_{i=0}^\infty \frac1{\lambda_i}=\infty$.
	
	Recíprocamente, supongamos que la serie $\sum_{i=0}^\infty \frac1{\lambda_i}$ diverge y veamos que $\smash{span_{C([0,1],\mathbb{C})}\{x^{\lambda_i}\mid i\in \mathbb{N}\cup \{0\} \}}$ es denso en $C([0,1],\mathbb{C})$.
	
	Dado un número real  $r\geq1$ tenemos que: $$\left|x^r-\sum_{i=i_0}^{n}a_ix^{\lambda_i}\right|=\left|\int_{0}^{x}\left(rt^{r-1}-\sum_{i=i_0}^{n}a_i\lambda_it^{\lambda_i-1}\right) dt  \right|\leq \int_{0}^{1}\left|rt^{r-1}-\sum_{i=i_0}^{n}a_i\lambda_it^{\lambda_i-1}\right|dt $$ $$\leq\sqrt{\int_{0}^{1}\left(rt^{r-1}-\sum_{i=i_0}^{n}a_i\lambda_it^{\lambda_i-1}\right)^2dt} =||rx^{r-1}-\sum_{i=i_0}^{n}a_i\lambda_ix^{\lambda_i-1}||_{L^2[0,1]} $$ donde en la primera igualdad hemos usado la regla de Barrow y en la segunda desigualdad hemos usado la desigualdad de Cauchy-Schwarz.
	
	Como hemos visto al principio de la demostración, la hipótesis de que  $\sum_{i=0}^\infty \frac1{\lambda_i}=\infty$ implica que para cada  $\smash{q\in \mathbb{N}\setminus \{\lambda_i \mid i\in\mathbb{N}\} }$ sea  $$x^q\in \overline{span_{L^2[0,1]}\{x^{\lambda_i}\mid i\in \mathbb{N}\cup \{0\} \}}.$$ 
	
	Si tomamos $i_0\in\mathbb{N}$ tal que $\lambda_i-1>0$ para cada $i>i_0$, entonces como también $\sum_{i=i_0}^\infty \frac1{\lambda_i-1}=\infty$ tendremos que $$x^q\in \overline{span_{L^2[0,1]}\{x^{\lambda_i-1}\mid i=i_0,i_0+1,\ldots \}}.$$

	Consideramos el caso en que $\smash{q-1\in \mathbb{N}\setminus \{\lambda_i-1 \mid i=i_0,i_0+1,\ldots\} }$, pues el otro es trivial. 
	
	Dado $\varepsilon>0$, podemos encontrar una función $$f\in span_{L^2[0,1]}\{x^{\lambda_i-1}\mid i=i_0,i_0+1,... \}$$ de modo que se cumpla la desigualdad $$ || qx^{q-1}-f||_{L^2[0,1]}<\varepsilon.$$
	Entonces, teniendo en cuenta la desigualdad que hemos probado previamente; recordamos, si $r\geq1$ entonces para cada $s\in[0,1]$:  $$\left|s^r-\sum_{i=i_0}^{n}a_is^{\lambda_i}\right|\leq||rx^{r-1}-\sum_{i=i_0}^{n}a_i\lambda_ix^{\lambda_i-1}||_{L^2[0,1]}, $$ 
	
	\noindent y considerando una primitiva $F$ de $f$, si tomamos supremos en $s\in[0,1]$ concluimos que $$ ||x^q-F||_{\infty}<\varepsilon$$ luego queda probado que también $$x^q\in \overline{span_{C([0,1],\mathbb{C})}\{x^{\lambda_i}\mid i\in \mathbb{N}\cup \{0\} \}},$$ lo cual junto al teorema de Weierstrass implica que $$C([0,1],\mathbb{C})= \overline{span_{C([0,1],\mathbb{C})}\{x^{\lambda_i}\mid i\in \mathbb{N}\cup \{0\} \}}.$$
\end{proof}

Observar que el Corolario \ref{potencias de n} se obtiene también como consecuencia de este Teorema.

\begin{proof}[2ª Demostración de la suficiencia]

Una de las ideas implícitas en la demostración anterior es que por el teorema de Weierstrass, para ver que $span_{C([0,1],\mathbb{C})}\{x^{\lambda_i}\mid i\in \mathbb{N}\cup \{0\} \}$ es denso en $C([0,1],\mathbb{C})$, basta con probar que para todo $q\in \mathbb{N}\setminus \{\lambda_i \mid i\in\mathbb{N}\}$ se tiene que $x^q\in \overline{span_{C([0,1],\mathbb{C})}\{x^{\lambda_i}\mid i\geq0 \}}$.

Supongamos que $\sum_{i=0}^\infty \frac1{\lambda_i}=\infty.$

Sea $Q_0(x):=x^p$ y definimos recursivamente $Q_n(x)$ en $0\leq x\leq 1$ como: $$Q_n(x):=(\lambda_n-q)x^{\lambda_n}\int ^1 _x Q_{n-1}(t)t^{-1-\lambda_n} \ dt \quad \forall n\geq 1.$$

Afirmamos que para cada $n\geq0$ podemos encontrar unos coeficientes $a_{n,i}\in\mathbb{R}$ tales que: $$Q_n(x)=x^q - \sum_{i=0}^n a_{n,i}x^{\lambda_i}.$$

Lo probamos por inducción, siendo el caso $n=0$ trivial. Supongamos que es cierto para $n$, entonces:  $$Q_{n+1}(x):=(\lambda_{n+1}-q)x^{\lambda_{n+1}}\int ^1 _x (x^q-\sum_{i=0}^n a_{n,i}x^{\lambda_i})t^{-1-\lambda_{n+1}} \ dt $$

e integrando y usando la regla de Barrow nos queda: $$Q_{n+1}(x):=(\lambda_{n+1}-q)x^{\lambda_{n+1}}\left[ \frac{t^{q-\lambda_{n+1}}}{q-\lambda_{n+1}} - \sum_{i=0}^n a_{n,i}\frac{t^{\lambda_i-\lambda_{n+1}}}{\lambda_i-\lambda_{n+1}} \right]^1_x = $$

$$=x^q +\sum_{i=0}^n a_{n,i}x^{\lambda_i}\frac{\lambda_{n+1}-q}{\lambda_{n+1}-\lambda_i} + x^{\lambda_{n+1}}(\lambda_{n+1}-q)\left(\frac1{\lambda_{n+1}-q}- \sum_{i=0}^n a_{n,i}\frac1{\lambda_i-\lambda_{n+1}} \right).$$

Por lo tanto, esto es cierto para todo $n$. Además se tiene que $\| Q_0\|_{\infty}=1$ y que $$| Q_n|\leq | \lambda_n-q|\|Q_{n-1}\|_{\infty}|x^{\lambda_n}|\left|\int ^1 _x t^{-1-\lambda_n} \ dt\right |\leq | \frac{\lambda_n-q}{\lambda_n}|\|Q_{n-1}\|_{\infty}|x^{\lambda_n}-1| $$

Lo cual implica, tomando supremos y aplicando esta desigualdad recursivamente, que: $$||Q_n||_{\infty}\leq \prod_{i=0}^n\left| 1-\frac{q}{\lambda_i}\right|.$$

Finalmente, por la Proposición \ref{prop-prod}, hemos demostrado que si $\sum_{i=0}^\infty \frac1{\lambda_i}=\infty$ entonces $\|Q_n\|_{\infty}\rightarrow 0$, con lo cual $x^q\in \overline{span_{C([0,1],\mathbb{C})}\{x^{\lambda_i}\mid i\in \mathbb{N}\cup \{0\} \}}.$

\end{proof}

\underline{Observación:} A diferencia de la demostración original de Müntz, esta tiene un carácter constructivo, y de hecho, la misma construcción nos habría servido para probar el Lema \ref{vabs} tomando $q:=1$ y $\lambda_i:=2i$ para cada $i\in\mathbb{N}$.

\subsection{Aplicaciones}

El siguiente resultado fue probado por primera vez por Euler en 1737. 

La prueba que incluimos será "semi-original", en el sentido de que comenzaremos el esquema de una demostración existente, la cual no solo prueba la divergencia de la serie si no que también exhibe su carácter asintótico, para poder probar cierta desigualdad, pero una vez probada, continuamos la demostración independientemente usando una Proposición que hemos probado en la anterior sección.

Así, podría mejorarse este resultado (dando más información sobre el comportamiento de la serie), pero para nuestros propósitos será más que suficiente.
 
\begin{teo}[Euler]\label{Euler}
	La serie de los recíprocos de los números primos diverge: $$\sum_{p\ primo}\frac{1}{p}=\infty.$$
\end{teo}
\begin{proof}[Demostración]
	
	Probamos primero que para cada $n\in\mathbb{N}$ se cumple la siguiente desigualdad: \begin{equation}\label{desigualdad-series}
	\sum_{i=1}^n\frac{1}{i}\leq\left(\prod_{p\leq n}\left(1+\frac{1}{p}\right)\right)\left(\sum_{i=1}^n\frac{1}{i^2}\right), 
	\end{equation}
	donde el producto se extiende sobre los primos $p$ menores o iguales que $n$.
	
	Escribimos $n$ como: $$n=p_1^{\alpha_1}\ldots p_r^{\alpha_r}p_{r+1}^{\alpha_{r+1}}\ldots p_s^{\alpha_s}$$ siendo $p_i$ primos distintos, los primero $r$ exponentes números impares y los demás pares. Entonces $$n=(p_1\ldots p_r)(p_1^{\beta_1}\ldots p_s^{\beta_s})^2$$ siendo $\beta_i=\frac{\alpha_i-1}{2}$ si $i=1,\ldots,r$ y $\beta_i=\frac{\alpha_i}{2}$ si $i=r+1,\ldots,s$.
	
	Así, si $1\leq i\leq n$, tendremos que existirán $p_1,\ldots,p_m$ primos distintos (y menores que $i$) y un natural $b\in\mathbb{N}$ tal que $$\frac{1}{i}=\frac{1}{p_1\ldots p_m\cdot b^2}=\left(\frac{1}{p_1\ldots p_m}\right)\left(\frac{1}{b^2}\right).$$
	
	Por otro lado, si desarrollamos el miembro de la derecha en (\ref{desigualdad-series}) y denotamos por $p_k$ al mayor primo menor que $n$, tendremos que:
	
	 $$
	\left(\prod_{p\leq n}\left(1+\frac{1}{p}\right)\right)\left(\sum_{i=1}^n\frac{1}{i^2}\right)=
	$$ $$ \left(1+\frac{1}{p_1}+\ldots+\frac{1}{p_k}+\frac{1}{p_1p_2}+\ldots\frac{1}{p_1p_k}+\ldots+\frac{1}{p_1\ldots p_k}\right)\left(1+\frac{1}{2^2}+\ldots+\frac{1}{n^2}\right)$$
	
	y de aquí se deduce trivialmente la desigualdad (\ref{desigualdad-series}), pues cada término de la forma $\frac{1}{i}=\left(\frac{1}{p_1\ldots p_m}\right)\left(\frac{1}{b^2}\right)$ pertenece a la anterior suma cuando desarrollamos ese producto.
	
	Es bien sabido que $$\sum_{i=1}^\infty\frac{1}{i^2}=\frac{\pi^2}{6},$$ y por ello (sería suficiente para el argumento con saber que la anterior serie es convergente) se cumple para cada $n\in\mathbb{N}$ que $$\left(\prod_{p\leq n}\left(1+\frac{1}{p}\right)\right) \geq \frac{6}{\pi^2} \left(\sum_{i=1}^n\frac{1}{i}\right), $$ lo cual implica por el Criterio de Comparación y la divergencia de la serie armónica que: $$ \prod_{p\leq n }\left(1+\frac{1}{p}\right)\rightarrow\infty \quad\mbox{ cuando }\quad n\rightarrow\infty.$$
	
	Ahora bien, puesto que: $$\prod_{p\leq n}\left(1+\frac{1}{p}\right) \rightarrow \infty\ \  \mbox{ si y sólo si }\ \ \prod_{p\leq n}\left(1+\frac{1}{p}\right)^{-1}=\prod_{p\leq n}\left(1-\frac{1}{p}\right) \rightarrow 0,  $$
	
	como consecuencia de la Proposición \ref{prop-prod} deducimos que  $$\sum_{p\ primo}\frac{1}{p}=\infty.$$
\end{proof}
\begin{cor}\label{span-primos}
	Se tiene que $span_{C([0,1],\mathbb{C})}\{1, x^p \mid p \mbox{ es un primo } \}$ es denso en $C([0,1],\mathbb{C})$
\end{cor}
\begin{proof}[Demostración]
	Es una consecuencia directa del Teorema de Müntz-Szász y del  Teorema \ref{Euler}.
\end{proof}
\begin{cor}
	Si $f\in C([0,1],\mathbb{C})$ cumple que $$\int_{0}^{1}f(x)x^p\ dx =0$$ para cada número primo $p$, entonces $f\equiv 0$.
\end{cor}
\begin{proof}[Demostración]
	
	Por el Corolario \ref{span-primos}, podemos encontrar una sucesión $(P_n)_{n=1}^\infty$ de polinomios formados con monomios cuyos exponentes son números primos tales que $P_n\rightarrow f$ uniformemente en $[0,1]$.
	
	Entonces $ (f(x)\cdot P_n(x))_{n=1}^\infty$ converge uniformemente a $[f(x)]^2$ en $[0,1]$.
	
	Ahora, si $P_n(x)=\sum_{i=i_0}^{n}a_ix^{p_i}$ para ciertos primos $p_i$ entonces por nuestra hipótesis $$\int_{0}^{1}f(x)P_n(x)\ dx =\sum_{i=i_0}^{n}a_i\int_{0}^{1}f(x)x^{p_i}\ dx =0$$ y como la convergencia de $ (f(x)\cdot P_n(x))_{n=1}^\infty$ es uniforme podemos intercambiar el límite y la integral y obtener que $$\int_{0}^{1}[f(x)]^2\ dx =\lim\limits_{n\rightarrow \infty}\int_{0}^{1}f(x)P_n(x)\ dx =0,$$ lo cual por la continuidad de $f$ implica que $f\equiv 0$.
	
\end{proof}

\section{Extensiones y generalizaciones}
La versión original del Teorema de Müntz puede resultar incompleta en cuanto a que hay restricciones sobre el conjunto de exponentes $\{\lambda_0\!=0\!<\lambda_1<\lambda_2<\ldots\}$. En 1996 Peter Borwein y Tamás Erdélyi probaron una versión de este teorema, ahora conocida como "Teorema de Müntz-Szász completo", en la cual no se considera restricción alguna sobre la sucesión de exponentes: 

\begin{teo}
	Sea $\{\lambda_n\}_{n=0}^\infty$ una sucesión de números reales positivos distintos entre sí. Entonces, $\smash{span_{C([0,1],\mathbb{C})}\{x^{\lambda_i}\mid i\geq0 \}}$ es denso en $C([0,1],\mathbb{C})$ si y sólo si $$\sum_{i=0}^\infty \frac{\lambda_i}{\lambda_i^2+1}=\infty.$$
\end{teo}
Además, en el mismo trabajo probaron versiones del Teorema de Müntz-Szász para los espacios $L^p[0,1]$. 

\vspace{1cm}

Comentamos que como consecuencia del Teorema de Weierstrass (\ref{weierstrass}) toda función $f\in C([0,1],\mathbb{C})$ puede ser aproximada uniformemente por polinomios con coeficientes racionales. Resulta natural pues, preguntarse si lo mismo sigue siendo cierto cuando consideramos polinomios con coeficiente enteros, y en caso de que no sea cierto qué condiciones adicionales tenemos que suponer que cumple la función $f$.

Consideramos el siguiente subespacio cerrado de $C([0,1],\mathbb{C})$: $$C_0([0,1],\mathbb{C}):=\{f\in C([0,1],\mathbb{C})\mid f(0),f(1)\in\mathbb{Z} \}.$$

En 1914 Kakeya probó que una función continua $f:[0,1]\rightarrow\mathbb{C}$ puede ser aproximada por polinomios con coeficientes enteros si y sólo si $f\in C_0([0,1],\mathbb{C})$. 

En 1975 Le Baron O. Ferguson y Manfred Von Golitschek combinaron los trabajos de Kakeya y de Müntz para probar los siguientes resultados:

\begin{teo}
	Sea $\{\lambda_0\!=0\!<\lambda_1<\lambda_2<\ldots\}$ una sucesión de números reales tales que $lim_{j\rightarrow\infty}\lambda_j=\infty$.
	Si $$\sum_{i=0}^\infty \frac1{\lambda_i}=\infty,$$ entonces toda función continua $f\in C_0([0,1],\mathbb{C})$ puede ser aproximada uniformemente por polinomios con coeficientes enteros de la forma: $$p(x)=\sum_{i=0}^{n}a_i x^{\lambda_i} \quad  (a_i\in\mathbb{Z}).$$

\end{teo}
\begin{teo}
	Sea $\{\lambda_i \}_{i=0}^\infty$ una sucesión de números reales positivos que posee un punto de acumulación en $x_0\in\mathbb{R}^+$, entonces toda función continua $f\in C_0([0,1],\mathbb{C})$ puede ser aproximada uniformemente por polinomios con coeficientes enteros de la forma: $$p(x)=\sum_{i=0}^{n}a_i x^{\lambda_i} \quad  (a_i\in\mathbb{Z}).$$
\end{teo}

\end{document}